\UseRawInputEncoding
\documentclass[12pt, twoside]{article}
\usepackage{amsmath,amsthm,amssymb}
\usepackage{times}
\usepackage{enumerate}
\usepackage{extarrows}
\usepackage{color}
\usepackage{tabularx}
\usepackage{multirow}

\usepackage {graphics}
\usepackage{graphicx}
\usepackage{url}
\def\titlerunning#1{\gdef\titrun{#1}}
\makeatletter
\def\author#1{\gdef\autrun{\def\and{\unskip, }#1}\gdef\@author{#1}}
\def\address#1{{\def\and{\\\hspace*{18pt}}\renewcommand{\thefootnote}{}%
		\footnote {#1}}%
	\markboth{\autrun}{\titrun}}
\makeatother
\def\email#1{e-mail: #1}
\def\subjclass#1{{\renewcommand{\thefootnote}{}%
		\footnote{\emph{Mathematics Subject Classification (2010):} #1}}}
\def\keywords#1{\par\medskip
	\noindent\textbf{Keywords.} #1}

\newtheorem{theorem}{Theorem}[section]
\newtheorem{lemma}{Lemma}[section]
\newtheorem{definition}{Definition}[section]
\newtheorem{proposition}{Proposition}[section]
\newtheorem{remark}{Remark}[section]
\newtheorem{corollary}{Corollary}[section]

\newcommand{\R}{\mathbb{R}}
\newcommand{\T}{\mathbb{T}}

\newcommand{\Q}{\mathbb{Q}}
\numberwithin{equation}{section}

\newcommand{\PreserveBackslash}[1]{\let\temp=\\#1\let\\=\temp}
\newcolumntype{C}[1]{>{\PreserveBackslash\centering}p{#1}}
\newcolumntype{R}[1]{>{\PreserveBackslash\raggedleft}p{#1}}
\newcolumntype{L}[1]{>{\PreserveBackslash\raggedright}p{#1}}

\newcolumntype{I}{!{\vrule width 1pt}}
\newlength\savedwidth

\def\subjclass#1{{\renewcommand{\thefootnote}{}%
		\footnote{\emph{Mathematics Subject Classification (2020):} #1}}}

\begin{document}
	
	
	\baselineskip=15pt
	
	
	\titlerunning{  }
	
	\title{Time periodic and almost periodic viscosity solutions of contact Hamilton-Jacobi equations on $\mathbb{T}^n$}
	
	\author{Kaizhi Wang \and Jun Yan \and Kai Zhao}

	\date{\today}
	
	\maketitle

	\address{Kaizhi Wang: School of Mathematical Sciences, CMA-Shanghai, Shanghai Jiao Tong University, Shanghai 200240, China; \email{kzwang@sjtu.edu.cn}
		\and Jun Yan: School of Mathematical Sciences, Fudan University, Shanghai 200433, China;
		\email{yanjun@fudan.edu.cn}
		\and
		Kai Zhao:  School of Mathematical Sciences, Tongji University, Shanghai 200092, China;
		\email{zhaokai93@tongji.edu.cn}}

	\begin{abstract}
 This paper concerns with the time periodic viscosity solution problem for a class of evolutionary contact Hamilton-Jacobi equations with time independent Hamiltonians on the torus $\mathbb{T}^n$. Under certain suitable assumptions we 
 show that the equation  has a non-trivial $T$-periodic viscosity solution if and only if $T\in D$, where $D$ is a dense subset of $[0,+\infty)$. Moreover, we clarify the structure of $D$. As a consequence, we also study the existence of Bohr almost periodic viscosity solutions. 
		
	\end{abstract}

\tableofcontents

%
%

\section{Introduction and main results} 
\subsection{The motivation of this paper}
This paper concerns with the existence of non-trivial time periodic and almost periodic viscosity solutions of evolutionary contact  Hamilton-Jacobi equations with time independent Hamiltonians on the flat $n$-torus $\T^n$. The novelty here is that our research object is a class of Hamilton-Jacobi equations with {\it autonomous contact Hamiltonians on a manifold of high dimension}. In order to describe our result clearly, we first recall some known work closely related to ours.

\begin{itemize}
		\item Classical Hamilton-Jacobi equations (where Hamiltonians are defined on the cotangent bundle $T^*\mathbb{T}^n$ of $\mathbb{T}^n$):
	\begin{itemize}
	\item  If $F(x,p)$ is a $C^2$  Hamiltonian defined on $T^*\mathbb{T}^n$ satisfying Tonelli conditions with respect to the argument $p$, then as $t\to +\infty$ the viscosity solution of 
	\begin{align}\label{00}
		\partial_t u(x,t) +  F(x,\partial_x u(x,t))=0, \quad (x,t)\in \mathbb{T}^n \times [0,+\infty)
	\end{align}
	converges to $\pm \infty$, or  a viscosity solution of 
	\[
	F(x,\partial_xu(x))=0.
	\] 
	 This means equation \eqref{00} has no non-trivial time periodic  viscosity solutions. See \cite{Fat-b,R} for details.
	\end{itemize}
	\item Contact Hamilton-Jacobi equations (where Hamiltonians are defined on the extended cotangent bundle $T^*\mathbb{T}^n\times\mathbb{R}$  equipped with the canonical contact form):
	\begin{itemize}
	\item Increasing case:  If $H(x,p,u)$ is a $C^3$ Hamiltonian defined on $T^*\mathbb{T}^n\times \mathbb{R}$
	and satisfies Tonelli conditions  with respect to the argument $p$  and $0\leq\frac{\partial H}{\partial u}\leq \lambda$ for some $\lambda>0$, then
	\[
	\partial_t w(x,t) +  H(x,\partial_x w(x,t), w(x,t))=0, \quad (x,t)\in \mathbb{T}^n \times [0,+\infty)
	\]
	does not admit non-trivial time periodic viscosity solutions for a similar reason as \eqref{00}. See \cite{S} for details.
	\item  Decreasing case:  If the Hamiltonian $H(x,p,u)$ is strictly decreasing in the argument $u$, the situation will be very different. Assume there are constants $\kappa>0$, $\delta>0$ such that $$
	-\kappa\leqslant\frac{\partial H}{\partial u}(x, p,u) \leqslant -\delta <0 ,\quad \forall (x,p,u)\in \mathbb{S}\times\R\times\R.
	$$ In \cite{WYZ23}, the authors proved that 
	\begin{align}\label{11}
	\partial_t w(x,t) + H( x,\partial_x w(x,t),w(x,t) )=0,\quad  (x,t)\in \mathbb{S} \times [0,+\infty)
	\end{align}
	has infinitely many time periodic viscosity solutions with different periods under certain additional  assumption. The structure of the unit circle $\mathbb{S}$ played an essential role there.

	
	\end{itemize}
\end{itemize}

{\it Are there non-trivial time periodic viscosity solutions of contact Hamilton-Jacobi equations in high-dimensional case?} This is the motivation of the present paper. In \cite{WYZ23} we used the fact that the Aubry set consists of a periodic orbit of the corresponding contact Hamiltonian system to construct the time periodic viscosity solution of \eqref{11}. However, this is not the case in the setting of this paper. Hence, we need new ideas here. In addition, we did not study the existence of almost periodic viscosity solutions in \cite{WYZ23}. As far as we know, there is little known about  time almost periodic viscosity solutions of Hamilton-Jacobi equations.  Ishii \cite{I2000} studied the almost periodic homogenization of Hamilton-Jacobi equations, where the Hamiltonian $H(x,y,p)$ is almost periodic in the argument $y$. Panov \cite{P} studied the long time behaviour of viscosity solutions of the Cauchy problem for a class of first-order Hamilton-Jacobi equations with spatially almost periodic initial functions.

\subsection{Assumptions and main results}
From now on we will use the symbol $H$ to denote a $C^3$  contact    Hamiltonian
\begin{align*}
	H: \mathbb{R}^n\times\mathbb{R}&\to \mathbb{R}\\
	(p,u)&\mapsto H(p,u).
\end{align*}
In this paper we study the time periodic viscosity solution problem for contact Hamilton-Jacobi equations
\begin{equation}\label{eq:HJe}\tag{E}
	\partial_t w(x,t) +  H(\partial_x w(x,t), w(x,t))=0, \quad (x,t)\in \mathbb{T}^n \times [0,+\infty),
\end{equation}
under the following standing assumptions on the  $H$:
	\begin{itemize}
	\item [\bf(H1)] $\frac{\partial^2 H}{\partial p^2} (p,u)>0$ for each $(p,u)\in \R^n \times\R$;
	\item  [\bf(H2)] for each $u \in  \R$, $H(p,u)$ is superlinear in $p$;
	\item [\bf(H3)] there are constants $\kappa>0$, $\delta >0 $  such that
	$$
	-\kappa\leqslant\frac{\partial H}{\partial u}( p,u) \leqslant -\delta <0 ,\quad \forall (p,u)\in  \R^n\times\R.
	$$	
\end{itemize}

%
\begin{remark}\label{1}
	Assumptions (H1) and (H2) are the standard Tonelli conditions. (H3) implies that $H$ is strictly decreasing in the argument $u$ and there exists a unique constant $c\in \R $ such that 
	$$
	H(0,c)=0.
	$$ 
\end{remark}

%

Denote by $\omega=(\omega_1,\cdots,\omega_n):= \frac{\partial H}{\partial p}(0,c)$
 and define 
$$
D:=\{|A|^{-1} : A=k_1\omega_1+ \cdots +k_n \omega_n \neq 0 , (k_1,\cdots, k_n)\in \mathbb{Q}^n \}.
$$
\begin{remark}\label{re2}
	It is easy to check that 
	\begin{align*}
		D=\Big \{T\in \R^+:   
		 (k_1\omega_1  + \cdots +k_n \omega_n) T=k_{n+1}\neq 0 \ \text{ for some }\ (k_1,\cdots,k_{n+1})\in \mathbb{Z}^{n+1}\Big\}.
	\end{align*}
It is clear that  $D$ is dense on $[0,+\infty)$ when $\omega\neq 0 $; $D= \emptyset$ when $\omega =0 $. 
\end{remark}

%
%

The first main result of this paper is stated as follows.
\begin{theorem}\label{thm:1} 
Equation  \eqref{eq:HJe}  
	 has a non-trivial $T$-periodic viscosity solution if and only if $T\in D$.
\end{theorem}

	In view of Theorem \ref{thm:1},  one can deduce that  equation \eqref{eq:HJe} has infinitely many non-trivial time periodic viscosity solutions whose periods differ from each other. From Remark \ref{re2} we know that the set of the periods is dense in $[0,+\infty)$.
Since any continuous non-constant periodic function has a fundamental period, a further discussion on the fundamental period problem is necessary.

Define
$$
\mathcal{D} :=\{|A|^{-1} : A=k_1\omega_1+ \cdots +k_n \omega_n \neq 0 , (k_1,\cdots, k_n)\in \mathbb{Z}^n   \}.
$$
Then $\mathcal{D}$ is dense on $[0,+\infty)$   if and only if there exist $i$, $j \in \{1,2,\cdots,n \}$  such that $\omega_i/\omega_j \notin \Q $.   

The second main result is the following.
\begin{theorem}\label{thm:2}
For any $T\in \mathcal{D}$, equation \eqref{eq:HJe} admits a non-trivial time-periodic viscosity solution with the fundamental period $T$. 
\end{theorem}

 As a direct consequence of Theorem \ref{thm:2}, by using the properties of almost periodic functions and viscosity solutions we get the following result.
\begin{corollary}\label{Cor:1}Equation \eqref{eq:HJe} admits infinitely many non-trivial Bohr almost periodic viscosity solutions if and only if there exist $i$, $j \in \{1,2,\cdots,n \}$  such that $\omega_i/\omega_j \notin \Q $, or equivalently, $\mathcal{D}$ is dense on $[0,+\infty)$.
\end{corollary}

\begin{remark}
	We recall the definition of Bohr almost periodic function here.
	Let $f:I\to \R$ be continuous with $I\subset\R$. 
	Given $\epsilon>0$, we call $\tau>0$ an $\epsilon$-period for $f$ if and only if  
	\[
	|f(t+\tau)-f(t)|\leqslant \epsilon,\quad t\in I.
	\]
	By $\mathcal{B}(f,\epsilon)$, we denote the set of all $\epsilon$-periods for $f$. We say that   
	$f$ is Bohr almost periodic if for any $\epsilon>0$, $\mathcal{B}(f,\epsilon)$ is  a relatively dense  in $[0,+\infty)$. See for instance \cite{Le} for more on almost periodic solutions of differential equations.
\end{remark}

The notion of the viscosity solution to Hamilton-Jacobi equations was introduced by Crandall and Lions in their pioneering work \cite{CL}. Since then, much
progress  in this field has been made by many authors. See for instance \cite{Bar94,Bar97,CIL,I,L,T}  for general overviews and developments of the theory of viscosity solutions. Ishii and Mitake \cite{IM} studied the existence of a time periodic solution of the Dirichlet
problem for a class of contact Hamilton-Jacobi equations
where the Hamiltonian is essentially time periodic.
As we mentioned before, in \cite{WYZ23} we studied the existence and multiplicity of viscosity solutions to a class of contact Hamilton-Jacobi equations on the unit circle. 
For time periodic viscosity solution problem for classical Hamilton-Jacobi equations, see for example \cite{BR,BS,FM,Mi,WY,WL} for results of different types.

The rest of this paper is organized as follows. In Section 2, we recall some definitions and results on contact Hamiltonian systems and translations. We also prove several preliminary results which will be used later. Section 3 is devoted to the proof of Theorem \ref{thm:1}. The proofs of Theorem \ref{thm:2} and Corollary \ref{Cor:1} will be given in Section 4， Section 5, respectively. We provide an example in the last part.

%

\section{Preliminaries}

\subsection{Weak KAM theory for contact Hamiltonian systems}
Consider the contact Hamiltonian system
\begin{equation} \label{eq:ode1}
	\left\{
	\begin{aligned}
		\dot x&=\frac{\partial H }{\partial p}(p,u),\\
		\dot p &=-\frac{\partial H }{\partial u}(p,u) \cdot p, \quad (p,u)\in \R^n \times \R, \\ 
		\dot u&=\langle\frac{\partial H }{\partial p}(p,u),p\rangle -H(p,u).
	\end{aligned}
	\right.
\end{equation}
The associated Lagrangian is defined by
$$
L(\dot x,u):=\sup_{p\in \R^n} \{ \langle \dot x, p\rangle -H(p,u) \},\quad (\dot{x},u)\in \R^n\times\R.
$$

Let us recall some known results on the Aubry-Mather-Ma\~n\'e-Fathi theory for \eqref{eq:ode1} here. Most of the results in this section can be found in  \cite{WWY,WWY1,WWY2,WWY3}.

\medskip

\noindent $\bullet$ {\bf Variational principles}. 
First recall implicit variational principles for contact Hamiltonian system (\ref{eq:ode1}), which connect contact Hamilton-Jacobi equations and contact Hamiltonian systems. 
\begin{proposition}(\cite[Theorem A]{WWY})\label{IVP}
	For any given $x_0\in \T^n$, $u_0\in\mathbb{R}$, there exists a continuous function $h_{x_0,u_0}(x,t)$  defined on $\T^n\times(0,+\infty)$ satisfying	
	\begin{align}\label{2-1}
		h_{x_0,u_0}(x,t)=u_0+\inf_{\substack{\gamma(0)=x_0 \\  \gamma(t)=x} }\int_0^tL\big(\dot{\gamma}(\tau),h_{x_0,u_0}(\gamma(\tau),\tau)\big)d\tau
	\end{align}
	where the infimum is taken among the Lipschitz continuous curves $\gamma:[0,t]\rightarrow \T^n$.
	Moreover, the infimum in (\ref{2-1}) can be achieved. 
	If $\gamma$ is a curve achieving the infimum \eqref{2-1}, then $\gamma$ is of class $C^1$.
	Let 
	\begin{align*}
		x(s)&:=\gamma(s),\quad u(s):=h_{x_0,u_0}(\gamma(s),s),\,\,\,\qquad  p(s):=\frac{\partial L}{\partial \dot{x}}(\dot{\gamma}(s),u(s)).
	\end{align*}
	Then $(x(s),p(s),u(s))$  satisfies equations \eqref{eq:ode1} with 
	\begin{align*}
		x(0)=x_0, \quad x(t)=x, \quad \lim_{s\rightarrow 0^+}u(s)=u_0.
	\end{align*}
\end{proposition}
We call $h_{x_0,u_0}(x,t)$ a forward implicit action function associated with $L$
and the curves achieving the infimum in (\ref{2-1}) minimizers of $h_{x_0,u_0}(x,t)$.

\medskip

\noindent $\bullet$ {\bf Implicit action functions and solution semigroups}. 
We recall some properties of implicit action functions first.

\begin{proposition}(\cite[Theorem C, Theorem D]{WWY}, \cite[Proposition 3.1, Proposition 3.4, Lemma 3.1]{WWY1} )\label{pr-af} \ 
	\begin{itemize}
		\item [(1)] (Monotonicity).
		Given $x_{0} \in M, u_{0}, u_{1}, u_{2} \in \mathbb{R}$, 
		\begin{itemize}
			\item [(i)] if $u_{1}<u_{2}$, then $h_{x_{0}, u_{1}}(x, t)<h_{x_{0}, u_{2}}(x, t)$, for all $(x, t) \in \T^n \times(0,+\infty)$;
			\item [(ii)] if $L_{1}<L_{2}$, then $h_{x_{0}, u_{0}}^{L_{1}}(x, t)<h_{x_{0}, u_{0}}^{L_{2}}(x, t)$, for all $(x, t) \in \T^n \times(0,+\infty)$
			where $h_{x_{0}, u_{0}}^{L_{i}}(x, t)$ denotes the forward implicit action function associated with $L_{i}, i=1,2 .$
		\end{itemize}
		\item [(2)] (Lipschitz continuity).
		The function $(x_0,u_0,x,t)\mapsto h_{x_0,u_0}(x,t)$ is  Lipschitz  continuous on $\T^n\times[a,b]\times \T^n\times[c,d]$ for all real numbers $a$, $b$, $c$, $d$
		with $a<b$ and $0<c<d$.
		\item [(3)] (Minimality).
		Given $x_0$, $x\in \T^n$, $u_0\in\mathbb{R}$ and $t>0$, let
		$S^{x,t}_{x_0,u_0}$ be the set of the solutions $(x(s),p(s),u(s))$ of (\ref{eq:ode1}) on $[0,t]$ with $x(0)=x_0$, $x(t)=x$, $u(0)=u_0$.
		Then
		\[
		h_{x_0,u_0}(x,t)=\inf\{u(t): (x(s),p(s),u(s))\in S^{x,t}_{x_0,u_0}\}, \quad \forall (x,t)\in \T^n\times(0,+\infty).
		\]
		\item [(4)] (Markov property).
		Given $x_0\in \T^n$, $u_0\in\mathbb{R}$, 
		\[
		h_{x_0,u_0}(x,t+s)=\inf_{y\in \T^n}h_{y,h_{x_0,u_0}(y,t)}(x,s)
		\]
		for all  $s$, $t>0$ and all $x\in \T^n$. Moreover, the infimum is attained at $y$ if and only if there exists a minimizer $\gamma$ of $h_{x_0,u_0}(x,t+s)$ with $\gamma(t)=y$.
		\item [(5)] (Reversibility).
		Given $x_0$, $x\in \T^n$ and $t>0$, for each $u\in \mathbb{R}$, there exists a unique $u_0\in \mathbb{R}$ such that
		\[
		h_{x_0,u_0}(x,t)=u.
		\]
	\end{itemize}
\end{proposition}

\medskip

Let us recall a  semigroup of operators introduced in \cite{WWY1}.  Define a family of nonlinear operators $\{T^-_t\}_{t\geqslant 0}$ from $C(\T^n,\mathbb{R})$ to itself as follows. For each $\varphi\in C(\T^n,\mathbb{R})$, denote by $(x,t)\mapsto T^-_t\varphi(x)$ the unique continuous function on $ (x,t)\in \T^n\times[0,+\infty)$ such that
\begin{equation*}\label{eq:semigroup}
	T^-_t\varphi(x)=\inf_{\gamma}\left\{\varphi(\gamma(0))+\int_0^tL(\dot{\gamma}(\tau),T^-_\tau\varphi(\gamma(\tau)))d\tau\right\},
\end{equation*}
where the infimum is taken among the absolutely continuous curves $\gamma:[0,t]\to \T^n$ with $\gamma(t)=x$.  It was also proved  in \cite{WWY1} that $\{T^-_t\}_{t\geqslant 0}$ is a semigroup of operators and the function $(x,t)\mapsto T^-_t\varphi(x)$ is a viscosity solution of $\partial_tw+H(x,\partial_xw,w)=0$ with initial condition $w(x,0)=\varphi(x)$. Thus, we call $\{T^-_t\}_{t\geqslant 0}$ the backward solution semigroup.


We collect some basic properties of the solution semigroup. 

\begin{proposition}(\cite[Propositions 4.1, 4.2, 4.3]{WWY1})\label{pr-sg}
	Let $\varphi$, $\psi\in C(\T^n,\R)$. 
	\begin{itemize}
		\item [(1)](Monotonicity). If $\psi<\varphi$, then $T^{-}_t\psi< T^{-}_t\varphi$, $\forall t\geqslant 0$.
		\item [(2)](Local Lipschitz continuity). The function $(x,t)\mapsto T^{-}_t\varphi(x)$ is locally Lipschitz on $\T^n\times (0,+\infty)$.
		\item[(3)]($e^{\kappa t}$-expansiveness). $\|T^{-}_t\varphi-T^{-}_t\psi\|_\infty\leqslant e^{\kappa t}\cdot\|\varphi-\psi\|_\infty$,  $\forall t\geqslant 0$.
		\item[(4)] (Continuity at the origin). $\lim_{t\rightarrow0^+}T^{-}_t\varphi=\varphi$.
		\item[(5)] (Representation formula). For each  $\varphi\in C(\T^n,\R)$,	
			$T^-_t\varphi(x)=\inf_{y\in \T^n}h_{y,\varphi(y)}(x,t)$,  $\forall (x,t)\in \T^n\times(0,+\infty)$.
		\item[(6)] (Semigroup). $\{T^{-}_t\}_{t\geqslant 0}$  is a one-parameter semigroup of operators. For all $x_0$, $x\in \T^n$, all $u_0\in\mathbb{R}$ and all $s$, $t>0$, 	 
		$$
			T^-_sh_{x_0,u_0}(x,t)=h_{x_0,u_0}(x,t+s), \quad T^-_{t+s}\varphi(x)=\inf_{y\in \T^n}h_{y,T^-_s\varphi(y)}(x,t).
		$$
	\end{itemize}
\end{proposition}


\medskip

\noindent $\bullet$ {\bf Weak KAM solutions}. 
Following Fathi (see, for instance, \cite{Fat-b}), one can define weak KAM solutions of  
\begin{align}\label{wkam}
	H(\partial_xu(x),u(x))=0.	
\end{align}

\begin{definition}
	\label{bwkam}
	A function $u\in C(\T^n,\R)$ is called a backward weak KAM solution of \eqref{wkam} if
	\begin{itemize}
		\item [(1)] for each continuous piecewise $C^1$ curve $\gamma:[t_1,t_2]\rightarrow \T^n$, we have
		\begin{align}\label{do}
			u(\gamma(t_2))-u(\gamma(t_1))\leqslant\int_{t_1}^{t_2}L(\dot{\gamma}(s),u(\gamma(s)))ds;
		\end{align}
		\item [(2)] for each $x\in \T^n$, there exists a $C^1$ curve $\gamma:(-\infty,0]\rightarrow M$ with $\gamma(0)=x$ such that
		\begin{align}\label{cali1}
			u(x)-u(\gamma(t))=\int^{0}_{t}L(\dot{\gamma}(s),u(\gamma(s)))ds, \quad \forall t<0.
		\end{align}
	\end{itemize}
	
	Similarly, 	a function $v\in C(\T^n,\R)$ is called a forward weak KAM solution of  \eqref{wkam} if it satisfies (1) and  
	for each $x\in \T^n$, there exists a $C^1$ curve $\gamma:[0,+\infty)\rightarrow \T^n$ with $\gamma(0)=x$ such that
	\begin{align}\label{cali2}
		v(\gamma(t))-v(x)=\int_{0}^{t}L(\dot{\gamma}(s),v(\gamma(s)))ds,\quad \forall t>0.
	\end{align}
	We say that $u$ in \eqref{do} is a dominated function by $L$. We call curves satisfying \eqref{cali1} (resp. \eqref{cali2}),  $(u,L,0)$-calibrated curves (resp. $(v,L,0)$-calibrated curves). We use $\mathcal{S}_-$ (resp. $\mathcal{S}_+$) to denote the set of all backward (resp. forward) weak KAM solutions. Backward weak KAM solutions and viscosity solutions are the same in the setting of this paper.
\end{definition}


\begin{proposition}{\rm{\cite[Theorem 1.2]{WWY1}}}\label{prop:liminf}
	Let $\varphi\in C(\T^n,\R)$. If the function $(x,t)\mapsto T_t^- \varphi(x)$ is bounded on $\T^n \times [0,+\infty)$, then  
	$ \displaystyle \liminf_{t\to +\infty} T^-_t \varphi(x) $ is a  viscosity solution of equation \eqref{wkam}. 
\end{proposition}

Let $\widetilde\Phi^H_t$ denote the local flow of \eqref{eq:ode1}.
\begin{definition}(\cite[Definitions 3.1, 3.2]{WWY2})\label{gl}
	\begin{itemize}
		\item[(1)] (Globally minimizing orbits)
		A curve $(x(\cdot),u(\cdot)):\R \to \T^n \times \R $ is called globally minimizing , if it is locally Lipschitz and for each $t_1<t_2\in\R$, there holds 
		$$
		u(t_2)=h_{x(t_1),u(t_1)}(x(t_2),t_2-t_1 ).
		$$
		\item[(2)](Static curves) A curve $(x(\cdot),u(\cdot)):\R \to \T^n \times \R $ is called static, if it is globally minimizing and for each $t_1,t_2\in \R$, there holds
		$$
		u(t_2)=\inf_{s>0} h_{x(t_1),u(t_1)}(x(t_2),s).
		$$
		If a curve $(x(\cdot),u(\cdot)):\T^n \times \R$ is static, then $(p(t),u(t))$ with $t\in \R$  is an orbit of $\widetilde\Phi^H_t$, where $p(t)=\frac{\partial L}{\partial \dot{x}}(\dot x(t),u(t )) $. We call it a static orbit of $\widetilde\Phi^H_t$.
	\end{itemize}
\end{definition}

\begin{definition}(\cite[Definition 3.3]{WWY2})\label{audeine}
	We call the set of all static orbits Aubry set of $H$, denoted by $\tilde{\mathcal{A}}$. We call $\mathcal{A}:=\mathrm{proj}(\tilde{\mathcal{A}})$ the projected Aubry set, where $\mathrm{proj}:T^*\T^n\times\R\rightarrow \T^n$ denotes the canonical projection. 
\end{definition}

Note that $u_0\equiv c$ is a classical solution  of \eqref{wkam}, where $c$ is as in Remark \ref{1}.
Under the assumptions imposed on $H$, the Aubry set of \eqref{eq:ode1} has the form  $\{(x,p,u);x\in \T^n,p=0,u=c\}$. The projected Aubry set is $\mathbb{T}^n$ on which the flow is linear.

\begin{proposition}(\cite[Theorem 2]{WWY3})\label{prop:WWY-longtime-JDE}
	Denote by $w_\varphi(x,t)$ the unique viscosity solution of \eqref{eq:HJe} with $w_\varphi(x,0)=\varphi(x)$, and by $u_+$ the unique forward weak KAM solution of \eqref{wkam}. Then
	\begin{itemize}
		\item [(D1)]$w_\varphi(x,t)$ is bounded on $\T^n \times[0,+\infty)$ if and only if 
		$\varphi \geq u_+$ everywhere and 
		there exists $x_{0} \in \T^n$ such that $\varphi\left(x_{0}\right)=u_+\left(x_{0}\right)$.
		In this case, for each $\varphi$ satisfying the above two conditions, there is $T_{\varphi}>0$ such that
		\[
		|w_\varphi(x,t)|\leq K,\quad\forall  (x,t)\in \T^n\times(T_{\varphi},+\infty)
		\]
		for some $K>0$ independent of $\varphi$.
		\item [(D2)]if there is $x_{0} \in \T^n$ such that $\varphi\left(x_{0}\right)<u_+\left(x_{0}\right)$, then $\lim _{t \rightarrow+\infty} w_\varphi(x,t)=-\infty$ uniformly on $x \in \T^n$.
		\item [(D3)] if $\varphi>u_{+}$ everywhere, then $\lim _{t \rightarrow+\infty} w_\varphi(x,t)=+\infty$ uniformly on $x \in \T^n$.
	\end{itemize}
	
\end{proposition}

\subsection{Two key lemmas}
For any $x_0:=( x^0_1, x^0_2,\cdots x^0_n)\in \mathbb{T}^n$ and any $\epsilon>0$, define a function on $\T^n\times[0,+\infty)$ by
  \begin{equation}\label{eq:w}
  \begin{split}
  	 W^\epsilon_{x_0}(x,t)=  c+\epsilon \sum_{i=1}^n \Big( 1+\sin   f_i(x_i,t) \Big)  ,\quad
   f_i(x_i,t) 	:=2\pi x_i- 2\pi x^0_i-\frac{\pi}{2} -2\pi  \omega_i t,  
  \end{split} 	 
  \end{equation}
where $c$ is as in Remark \ref{1}.
\begin{lemma}\label{lem:subsolution-w}
	For any $x_0\in \mathbb{T}^n$, there is $\epsilon_0>0$ such that $W^\epsilon_{x_0}(x,t)$ is a subsolution of \eqref{eq:HJe} for any $\epsilon \in  (0,\epsilon_0 ]$. Moreover, for any $(x,t)\in \mathbb{T}^n\times[0,+\infty)$, 
	$$
	W^\epsilon _{x_0}(x,t) \geqslant c , \quad  \{x\in \mathbb{T}^n:W^\epsilon_{x_0}(x,0)=c \} =\{x_0\}, \quad h_{x_0,c}(x,t) \geqslant W^\epsilon_{x_0}(x,t).
	$$
\end{lemma}

\begin{proof}
It is clear that $ W^\epsilon _{x_0}(x,t) \geqslant c$. 	For any $\nu:=(\nu_1,\cdots , \nu_n) \in[0,1)^n$, let
		\begin{equation*}\label{eq:h_pp}
			\widehat H^\nu _{p_ip_j}(x,t):=\int_0^1 s\int_0^1 \frac{\partial^2 H}{\partial p_i\partial p_j} \Big (x, \nu_1 s \tau   \cos f_1(x_1,t), \cdots, \nu_n s \tau   \cos f_n(x_n,t)  , c  \Big )  d\tau \ ds.  		
		\end{equation*}
			Let
		$$
		M_0:=\max_{\substack{ (x,t)\in \mathbb{T}^n\times [0,+\infty) \\
				\nu\in[0,1)^n ,i,j\in\{1,\cdots,n \}}} n \cdot \Big|\widehat H^\nu_{p_ip_i}(x,t) \Big |.
		$$ 	
		Take	
		\[
		\epsilon_0: =\min\Big\{\frac{1}{2} \delta  \Big(4\pi^2 M_0 \max_{1\leqslant i\leqslant n} \omega^2_i \Big)^{-1},1\Big\},
		\]
		where $\delta$ is as in (H3).
		
		We assert that 	
		$W^\epsilon _{x_0}(x,t)$ defined in \eqref{eq:w} is a subsolution of \eqref{eq:HJe} with $W^\epsilon_{x_0}(x_0,0)=c$.
		For any $(x,t)\in \mathbb{T}^n\times[0,+\infty)$, by direct computation we have that
		\begin{align*}
			&\, \partial_t W^\epsilon _{x_0} (x,t)+H(\partial_x W^\epsilon _{x_0} (x,t), W^\epsilon _{x_0} (x,t))\\
			\leqslant &\,  \partial_t W^\epsilon _{x_0} (x,t) + H( \partial_x W^\epsilon _{x_0} (x,t) ,c) -\delta (W^\epsilon_{x_0}(x,t) -c ) \\
			\leqslant &\,  \partial_t W^\epsilon _{x_0} (x,t) + H( 0 ,c) + \left\langle\frac{\partial H}{\partial p} (0,c), \partial_x W^\epsilon _{x_0} (x,t)\right\rangle+ M_0  \left\langle  \partial_x W^\epsilon _{x_0} (x,t) , \partial_x W^\epsilon _{x_0} (x,t) \right\rangle\\  &\  -\delta  (W^\epsilon_{x_0}(x,t) -c ) \\
		= &\, - \epsilon \sum_{i=1}^n 2\pi \omega_i  \cos f_i(x_i,t)  + \sum_{i=1}^n \omega_i  2\pi\epsilon  \cos f_i(x_i,t)   + M_0 4\pi^2\epsilon^2  \sum_{i=1}^n    \omega_i^2  \cos^2 f_i(x_i,t) \\  &\ -\delta  \epsilon \sum_{i=1}^n \Big( 1+\sin f_i(x_i,t) \Big)    \\
			= &\,  M_0 4\pi^2\epsilon^2  \sum_{i=1}^n    \omega_i^2 \cos^2 f_i(x_i,t)  -\delta  \epsilon \sum_{i=1}^n \Big( 1+\sin f_i(x_i,t) \Big)  \\
				\leqslant &\, \epsilon \delta \sum_{i=1}^n\Big( \frac{1}{2} \cos^2 f_i(x_i,t) -1 -\sin f_i(x_i,t)   \Big)\\
			=  &\, -\epsilon \delta \frac{1}{2} \sum_{i=1}^n \Big(1 + \sin f_i(x_i,t) \Big)^2 \leqslant 0.
		\end{align*}
		So far, we have proved that $W_{x_0}^\epsilon (x,t)$ is a subsolution of \eqref{eq:HJe}.
		
From the classical comparison principle, one can deduce that
		$T^-_t W^\epsilon _{x_0}(x,0) \geqslant  W^\epsilon _{x_0}(x,t)$, since $T^-_t W^\epsilon _{x_0}(x,0)$ is a viscosity solution of \eqref{eq:HJe} while $W^\epsilon _{x_0}(x,t)$ is a  subsolution of \eqref{eq:HJe}. Thus,   we get that
 \begin{align*}
	h_{x_0,c}(x,t)=h_{x_0,W^\epsilon _{x_0}(x_0,0)}(x,t)\geqslant    T_t^- W^\epsilon _{x_0}(x,0) \geqslant W^\epsilon_{x_0}(x,t)
	 \end{align*}
	 for any $(x,t)\in \T^n\times [0,+\infty)$.
\end{proof}
Denote by $\Pi:\mathbb{R}^n\to \mathbb{T}^n$ the standard universal covering projection, and by 
\begin{align*}
	  \Phi^H_t:  &\,\T^n\to \T^n \\
	   &\, x\mapsto \Pi(x+\omega t)
\end{align*}
the linear flow on $\T^n$.

\begin{lemma} \label{lem:lem-2}
For any $\varphi \in  C(\T^n,\R )$ with $\displaystyle \min_{x\in \T^n} \varphi(x)=c$, 
$$
T_t^- \varphi( \Phi^H_t(x_0))=h_{x_0,c}(\Phi^H_t(x_0),t)=c, \quad \forall t\geqslant 0,\ \forall x_0\in \arg \min_{x\in\T^n} \varphi(x).
$$
\end{lemma}
\begin{proof}
	Since $u_0\equiv c$ is a classical solution of equation \eqref{eq:HJe}, then 
	\begin{equation}\label{eq:pf-lem-2-1}
	T^-_t \varphi( x) \geqslant T^-_t  u_0(x) =u_0(x)\equiv c, \quad \forall x\in \T^n.
	\end{equation}
	On the other side, 
	\begin{equation}\label{eq:pf-lem-2-2}
	T^-_t \varphi(  \Phi^H_t(x_0))= \inf_{y\in \T^n} h_{y, \varphi(y) }  \left( \Phi^H_t(x_0),t\right)  \leqslant h_{x_0, \varphi(x_0) } \left( \Phi^H_t(x_0),t\right).
	\end{equation}
	 In view of \eqref{eq:pf-lem-2-1} and \eqref{eq:pf-lem-2-2}, we have 
	 \begin{equation}\label{eq:pf-lem-2-3}
	 	 h_{x_0, \varphi(x_0) }\left( \Phi^H_t(x_0),t\right) \geqslant c.
	 \end{equation}
	By Proposition \ref{pr-af} (3), we have
	$$
	h_{x_0, \varphi(x_0) }\left( \Phi^H_t(x_0),t\right)= \inf \Big \{u(t):(x(s),p(s),u(s))\in S^{\Phi^H_t(x_0),t}_{x_0, \varphi(x_0)}\Big \}.
	$$
 Since 
	$$
	\Big\{\widetilde \Phi_s^H(x_0,0,c),s\in[0,t ]\Big\} \in  S^{\Phi^H_t(x_0),t}_{x_0, \varphi(x_0)},
	 $$
 then we have 
 $$
 h_{x_0, \varphi(x_0) }\left( \Phi^H_t(x_0),t\right)\leqslant c, \quad \forall t\geqslant 0.
 $$
In view of \eqref{eq:pf-lem-2-2}, \eqref{eq:pf-lem-2-3} and the above inequality,  
\begin{equation*}\label{eq:pf-lem-2-4}
T_t^- \varphi( \Phi^H_t(x_0))\leqslant h_{x_0, \varphi(x_0) }\left( \Phi^H_t(x_0),t\right) =c,
\end{equation*}
which completes the proof.
\end{proof}

\subsection{Translations and linear flow on $\T^n$ }
We first recall some definitions and known results in this part. Definitions \ref{d1} and \ref{d2} and Propositions \ref{Appendix:1} and \ref{Appendix:2} can be found in \cite{KH}. We prove Proposition \ref{Appendix:3} at last.

\begin{definition}\label{d1}
	A topological dynamical system $f:X\to X$ is called \textit{topologically transitive} if there exists a point $x\in X$ such that its orbit $\mathcal{O}_f(x):= \{f^n(x)\}_{n\in \mathbb{Z}}$ is dense in X. 
\end{definition}
\begin{definition}\label{d2}
	A topological dynamical system $f:X\to X$ is called \textit{minimal} if the orbit of every point $x\in X$ is dense in X, or, equivalently, if $f$ has no proper closed invariant sets.
\end{definition}

Let $\gamma=(\gamma_1, \cdots \gamma_n)\in \mathbb{T}^n$. The translation $T_\gamma$ has the form
$$
\mathfrak{T}_\gamma(x_1,\cdots,x_n)=\Pi(x_1+\gamma_1,\cdots,x_n+\gamma_n ) 
$$
\begin{proposition}\label{Appendix:1}
	The translation $\mathfrak{T}_\gamma$  is minimal if and only if the numbers $\gamma_1,\cdots \gamma_n$ and 1 are rationally independent, that is, if $\displaystyle \sum_{i=1}^nk_i\gamma_i$ is not an integer for any collection of integers $k_1,\cdots,k_n$ expect for $k_1=k_2=\cdots=k_n=0$. 
\end{proposition}
Let us consider the following system of differential equations on $\T^n$ 
$$
\frac{d x_i}{d t}=  \omega_i, \quad i=1,\cdots,n.
$$
Again integration produces a one-parameter group of translations
$$
\mathfrak{T}_\omega^t(x_1,\cdots,x_n)=\Pi(x_1+t \omega_1, \cdots, x_n+t \omega_n ).
$$
\begin{proposition}\label{Appendix:2}
	The flow $\{\mathfrak{T}_\omega^t\}$ is minimal if and only if the numbers $\omega_1,\cdots,\omega_n$ are rationally independent, that is, if $\displaystyle \sum_{i=1}^n  k_i \omega_i \neq 0$ for any integers $k_1, \cdots k_n$ unless $k_1=\cdots =k_n=0$.
\end{proposition}

The following proposition will be useful in the proof of Theorem \ref{thm:1}.
\begin{proposition}\label{Appendix:3}
	Assume $T\notin D$. The translation $\mathfrak{T}_\omega$ is minimal on the orbit of the flow $\mathfrak{T}^t_\omega$, i.e., for any $x_0\in \T^n$,
	$$
	\bigcup_{n\in \mathbb{Z}}\Phi_{nT}^H(x_0)  \text{ is dense in }   \bigcup_{t\in \R} \Phi_{t}^H(x_0).
	$$
\end{proposition}

	\begin{proof}
 We only need to focus on the case $\omega\neq 0$. 
		Assume $\omega_1,\cdots \omega_m, 1\leqslant m<n$ is a maximum linearly independent group, then $\omega_1,\cdots \omega_m $ are rationally independent. Due to $ T\notin D$,  similarly with \eqref{eq:pf-thm1-step5--claim-0}, we have that $\omega_1T,\cdots, \omega_mT $ and $1$ are rationally independent.
		
		Applying Proposition \ref{Appendix:1}, the translation 
		\begin{align*}
			\mathfrak{T}_\omega|_{\T^m}: \quad\quad\quad \T^m \quad \quad \to &\, \quad\quad\T^m \\
		  (x^0_1,\cdots,x^0_m)\mapsto &\, \Pi(x^0_1+\omega_1,\cdots,x^0_m+\omega_n) 
		\end{align*}
		 is  minimal on $\T^m$. Applying Proposition \ref{Appendix:2}, the flow 
		\begin{align*}
			\mathfrak{T}^t_\omega|_{\T^m}: \quad\quad\quad \T^m \quad \quad \to &\, \quad\quad\T^m \\
		  (x^0_1,\cdots,x^0_m)\mapsto &\, \Pi(x^0_1+t\omega_1,\cdots,x^0_m+t\omega_n) 
		\end{align*}
		is minimal on $\T^m$. It follows that the translation $\mathfrak{T}_\omega|_{\T^m}$ is minimal on the orbit of flow $\mathfrak{T}^t_\omega|_{\T^m}$.
		
		For any $j=m+1,m+2,\cdots, n$, there exist $\{a_{jk}\}_{k=1}^m$ such that
		$$
		\omega_j= a_{j1} \omega_1+a_{j2} \omega_2+\cdots +a_{jm} \omega_m .
		$$
		This implies that for any $x\in \mathfrak{T}^t_\omega(x_0)$ and $\epsilon>0$, $ x_i=x^0_i+ t \omega_i$,  there exists  $N>0$ such that 
		$$
		\sum_{i=1}^m \min \Big\{ | x_i-  \{ x^0_i+ N\omega_i  \}  |, 1- | x_i-  \{ x^0_i+ N\omega_i  \}  | \Big\}\leqslant \delta:= \Big( (n-m)\cdot \max_{\substack{ 1\leqslant k\leqslant m \\ m+1\leqslant j\leqslant n }}|a_{jk} | +1\Big)^{-1}  \epsilon
		$$ 
		and
		\begin{align*}
			&\, \min \Big\{ | x_j-  \{ x^0_j+ N\omega_j  \}  |, 1- | x_j-  \{ x^0_j+ N\omega_i  \}  | \Big\} \\
			\leqslant &\,  \max_{1\leqslant k\leqslant m}|a_{jk} | \cdot  \sum_{i=1}^m \min \Big\{ | x_i-  \{ x^0_i+ N\omega_i  \}  |, 1- | x_i-  \{ x^0_i+ N\omega_i  \}  | \Big\} \\
			\leqslant &\,  \max_{1\leqslant k\leqslant m}|a_{jk} | \cdot \delta.
		\end{align*}
		Hence, we get that
		\begin{align*}
			 \sum_{i=1}^n \min \Big\{ | x_i-  \{ x^0_i+ N\omega_i  \}  |, 1- | x_i-  \{ x^0_i+ N\omega_i  \}  | \Big\} 
			\leqslant   \Big( (n-m)\cdot \max_{\substack{ 1\leqslant k\leqslant m \\ m+1\leqslant j\leqslant n }}|a_{jk} | +1\Big) \cdot \delta \leqslant \epsilon.
		\end{align*}
		This implies that the translation $\mathfrak{T}_\omega$ is   minimal on the orbit of flow $\mathfrak{T}^t_\omega|_{\T^n}$. 
	\end{proof}

\section{Proof of Theorem \ref{thm:1}}

\begin{proof}[Proof of Theorem \ref{thm:1}]
Since the proof is quite long, we divide it into five steps.
Steps 1-4 are devoted to the proof of the fact that for any $T\in D$ there is a non-trivial $T$-periodic viscosity solution of \eqref{eq:HJe}. In the last step we show that if equation \eqref{eq:HJe} has a $T$-periodic viscosity solution then $T$ must belong to the set $D$.

First, we show that for any $T\in D$ there is a non-trivial $T$-periodic viscosity solution of \eqref{eq:HJe}.
For any $T\in D$, there exists $( k_1,\cdots, k_{n+1})\in \mathbb{Z}^{n+1} $ such that 
	\begin{equation}\label{eq:pf-thm1-T-1}
		(k_1\omega_1+ \cdots +k_n \omega_n )T=k_{n+1} \neq 0.
	\end{equation}
Without any loss of generality, we assume $k_{n+1}=1$.
If  $k_{n+1}\neq 1$, notice that
$$
(k_1\omega_1+ \cdots +k_n \omega_n )\cdot k_{n+1}^{-1}T=1 \neq 0.
$$
A non-trivial $k_{n+1}^{-1}T$-periodic solution also is a non-trivial $T$-periodic solution.  So we only need to consider the case $k_{n+1}=1$.

\medskip

\noindent {\bf Step 1:}	We will show that for any $T\in D$, there exists a closed subset $S\subset \T^n$ satisfying
\begin{align}\label{eq:eq:pf-thm1-step1-1}
\Phi^H_T (S)=S, \quad \text{and} \quad \Phi^H_t (S) \cap S=\emptyset, \quad \forall \, t\in  (0,T).
\end{align}
Let $l:=\{x\in \R^n :  k_1x_1+ \cdots +k_n x_n \in \mathbb{Z} \} \subset \mathbb{R}^n$. Set $ S=\Pi(l)$.   
First, for any $x\in S$,  we have $k_1x_1+ \cdots +k_n x_n \in \mathbb{Z}$ and
\begin{align*}
	k_1(x_1\pm \omega_1 T)+ k_2(x_2 \pm \omega_2 T)+\cdots +k_n(x_n \pm\omega_n T)=k_1x_1+ \cdots +k_n x_n\pm1\in \mathbb{Z}.
\end{align*}
This implies that
\begin{align*}
	\Phi_T^H (x)= \Pi( x+ \omega T) \in S, \quad 	x=\Phi_T^H (\Pi(x-\omega T) )\in \Phi_T^H (S).
\end{align*}
Thus,  $\Phi^H_T (S)=S$.   Next, for any $t\in (0,T)$,  we get that
\begin{align*}
k_1(x_1+ \omega_1 t)+k_2(x_2+ \omega_2 t)+\cdots +k_n(x_n+ \omega_n t)=k_1x_1+k_2x_2+\cdots +k_n x_n+ \frac{t}{T} \notin \mathbb{Z},
\end{align*}
which means  $\Phi^H_t(S) \cap S=\emptyset$ for any $t\in (0,T)$. We have proved \eqref{eq:eq:pf-thm1-step1-1}.

\medskip

\noindent {\bf Step 2:} Define 
$$
U_k(x,t):=\inf_{y\in S}h_{y,c}(x,kT+t), \quad \forall (x,t)\in \mathbb{T}^n \times [0,+\infty).
$$
We will show that for any $T\in D$, $U_k(x,t)$ is decreasing  with respect to $k\in \mathbb{N}$ for any $(x,t)$ and $\{U_k(x,t)\}_k$ are uniformly bounded on $\T^n \times [0, +\infty)$.

First, we prove that $U_k(x,t)$ is decreasing with respect to $k\in \mathbb{N}$. For any given $x_0\in S$, since $x(t):=\Pi( x_0+\omega t)= \Phi^H_t (x_0) , u(t)\equiv c,t\in \R$ is globally minimizing, then for each $t_1<t_2\in\R$, there holds 
$$
u(x(t_2))=h_{x(t_1),u(x(t_1))}(x(t_2),t_2-t_1).
$$  
And thus, we get that
\begin{equation}\label{eq:pf-step2-0}
	 c=h_{x_0,c}(x(t),t)=h_{x_0,c}(\Phi^H_t (x_0) ,t),\quad \forall \, t>0.
\end{equation}
Note that for any $x_0\in S$, for any $(x,t)\in \mathbb{T}^n \times [0,+\infty)$,  
\begin{align*}
		h_{x_0,c}(x,(k+1)T+t) =&\, \inf_{z\in \mathbb{T}^n} h_{z,h_{x_0,c}(z,T)}(x,kT+t ) \\
		 \leqslant &\, h_{x(T),h_{x_0,c}(x(T),T)}(x,kT  +t)\\ =&\, h_{x(T),c }(x,kT+t).
\end{align*}
 
Since $x_0+ \omega T\in l$, then $x(T)=\Pi(x_0+ \omega T )\in S$. In view of \eqref{eq:eq:pf-thm1-step1-1}, it  follows that for any $(x,t)\in \mathbb{T}^n \times [0,+\infty)$,
	\begin{equation}\label{eq:pf-step2-1}
	\begin{split}
			U_{k+1}(x,t)=&\, \inf_{y\in S}h_{y,c}(x,(k+1)T+t) \\
			 \leqslant &\, \inf_{y\in S}  h_{\Pi(y+\omega T) ,c}(x,kT+t)\\
			   = &\,\inf_{y\in \Phi^H_{T}(S)}  h_{y,c}(x,kT+t)=U_{k}(x,t).
	\end{split}	
	\end{equation}

	Next, we prove that $U_k(x,t)$ is uniformly bounded. 	 Note that $u_0\equiv c$ is a solution of $H(\partial_x u,c)=0$. Thus
	\begin{equation}\label{eq:pf-thm1-step2-2}
	c=u_0(x)=T_t^- u_0(x)=\inf_{y\in \T^n} h_{y,u_0(y)}(x,t)= \inf_{y\in \T^n} h_{y,c}(x,t), \quad \forall x\in \T^n,t>0.
	\end{equation}
	Thus for any $(x,t)\in \T^n\times [0, +\infty)$, 
	\begin{align*}
		U_k(x,t)=\inf_{y\in S}h_{y,c}(x,kT+t) \geqslant c, \quad k\in \mathbb{N}.
	\end{align*}
	In view of \eqref{eq:pf-step2-1},  $U_k(x,t)$ is uniformly bounded.
	
\medskip

\noindent {\bf Step 3:} We show that for any $T\in D$,  $u(x,t):=\displaystyle \lim_{k\to +\infty } U_k(x,t)$ is   a  $T$-periodic viscosity solution of equation \eqref{eq:HJe}.

	By step 2, for any $(x,t)\in\T^n\times[0,+\infty)$, the following limit  exists 
	\begin{equation}\label{eq:eq:pf-thm1-step3-def-u}
	u(x,t):=\lim_{k\to +\infty } U_k(x,t)= \lim_{k\to +\infty }\inf_{y\in S}h_{y,c}(x,kT+t).
	\end{equation}
 Next, we claim that for any $t>0$,
\begin{equation}\label{eq:pf-thm1-step3-0}
		 	\lim_{k\to +\infty } U_k(x,t)= u(x,t), \quad \text{uniformly on } \T^n. 
		 \end{equation}
		 In fact, by Lemma \ref{lem:Tt-+inf} for any $t\geqslant 0$
		  \begin{equation}\label{eq:pf-thm1-step3-1}
	\begin{split}
		 U_k(x,t)=\inf_{y\in S}h_{y,c}(x,kT+t)= \inf_{y\in S} T_t^- h_{y,c}(x,kT) 
		 =T_t^-  \inf_{y\in S}  h_{y,c}(x,kT) =T_t^-  U_{k}(x,0),
	\end{split}
		 \end{equation}
		 and
  \begin{equation}\label{eq:pf-thm1-step3-2}
	\begin{split}
		 	 U_k(x,t)=&\, \inf_{y\in S}h_{y,c}(x,kT+t)\\
		 	 =&\, \inf_{y\in S} T_T^-h_{y,c}(x,(k-1)T+t)\\
		 	=&\,T_T^- \inf_{y\in S}  h_{y,c}(x,(k-1)T+t)\\ =&\, T_T^- U_{k-1}(x,t).
	\end{split}
		 \end{equation}
		  We show that the equi-Lipschitz property of $\{U_k(x,t) \}_{k>2}$ with respect to $x$ on  $\T^n$ for any fixed $t\in [0,+\infty)$. Denote by $K_2(t)>0$ a constant such that 
		 $\| U_k(\cdot ,t)  \|_\infty \leqslant K_2$ for all $k>1$. For  any $x,y\in \T^n$, by \eqref{eq:pf-thm1-step3-2}  we get that 
		 \begin{align*}
   |U_k(x,t)-U_k(y,t)|  =& | T_T^- U_{k-1}(x,t) - T_T^- U_{k-1}(y,t) | \\
    = & | \inf_{z \in M} h_{z,U_{k-1} (z,t)} (x,T) -  \inf_{z \in M} h_{z,U_{k-1} (z,t)} (y,T) | \\
     \leqslant &  \sup _{z \in M} | h_{z,U_{k-1} (z,t)} (x,T) - h_{z,U_{k-1} (z,t)} (y,T)|.
\end{align*}
Since $h_{.,.}(.,T)$ is uniformly Lipschitz on $\T^n \times [-K_2,K_2] \times \T^n$ with the Lipschitz constant $c_0 >0$, then we get that
$$
 |U_k(x,t)-U_k(y,t)|  \leqslant c_0 \ \|x-y\|, \quad \forall \, k >2.
$$
This implies \eqref{eq:pf-thm1-step3-0} holds true.

For any $x\in \mathbb{T}^n$, any $t\in [0,+\infty)$, by \eqref{eq:pf-thm1-step3-1} 	we obtain that
	$$
	u(x,t)=\lim_{k\to +\infty } U_k(x,t)= \lim_{k\to +\infty } T_t^-U_{k}(x,0)=T_t^- \lim_{k\to +\infty } U_{k}(x,0)=T_t^- u(x,0).
	$$
For any $x\in \mathbb{T}^n$, any  $t\in [0,+\infty)$,
	\begin{align*}
		T_T^- u(x,t)=&\, T_T^- \lim_{k\to +\infty }\inf_{y\in S}h_{y,c}(x,kT+t) \\
		=&\,  \lim_{k\to +\infty }T_T^-\inf_{y\in S}h_{y,c}(x,kT+t) \\
		=&\,  \lim_{k\to +\infty }\inf_{y\in S}h_{y,c}(x,(k+1)T+t)\\
		=&\, u(x,t).
	\end{align*}
	Hence $u(x,t)$ is  a  $T$-periodic viscosity solution of \eqref{eq:HJe}.

\medskip

\noindent {\bf Step 4:} We show that for any $T\in D$,  $u(x,t)$ is a  non-trivial $T$-periodic viscosity solution of equation \eqref{eq:HJe}.

	For any given $x_0\in S$,  by the  Lemma \ref{lem:subsolution-w}, for any $k\in \mathbb{N},y\in \T^n$, we have
	$$
	h_{y,c}(x_0,kT+\frac{T}{2}) \geqslant W^\epsilon_{y}(x_0,kT+\frac{T}{2}). 
	$$	
 In view of \eqref{eq:eq:pf-thm1-step1-1}, for any $t \geqslant 0$, we have 
	\begin{equation}\label{eq:pf-thm1-step4-1}
	\begin{split}
		\displaystyle \inf_{y\in S} W^\epsilon_{y}(x, t+T) = &\,	 c+\displaystyle \inf_{y\in S} \epsilon \sum_{i=1}^n \Big( 1+\sin \Big( 2\pi x_i- 2\pi y_i-\frac{\pi}{2} -2\pi  \omega_i (t+T) \Big)  \Big) \\
		=&\, c+\displaystyle \inf_{y\in S} \epsilon \sum_{i=1}^n \Big( 1+\sin  \Big( 2\pi x_i -2\pi (y_i+\omega_i T) -\frac{\pi}{2} -2\pi  \omega_i t \Big) \Big)\\
		=&\, c+\displaystyle \inf_{y\in S} \epsilon \sum_{i=1}^n \Big( 1+\sin  \Big( 2\pi x_i -2\pi\cdot \Pi(y_i+\omega_i T )-\frac{\pi}{2} -2\pi  \omega_i t \Big) \Big) \\
		=&\, c+\displaystyle \inf_{z  \in \Phi_{T}^H (S)} \epsilon \sum_{i=1}^n \Big( 1+\sin  \Big( 2\pi x_i -2\pi z_i -\frac{\pi}{2} -2\pi  \omega_i t \Big) \Big) \\
		=&\, c+\displaystyle \inf_{z  \in   S} \epsilon \sum_{i=1}^n \Big( 1+\sin  \Big( 2\pi x_i -2\pi z_i -\frac{\pi}{2} -2\pi  \omega_i t \Big) \Big)\\
		=&\, 	\displaystyle \inf_{y\in S} W^\epsilon_{y}(x, t) ,
	\end{split}
	\end{equation}
	showing that $\displaystyle \inf_{y\in S} W^\epsilon_{y}(x,\cdot )$ is $T$-periodic.  
 Thus,  
	\begin{align*}
		\lim_{k\to +\infty} \inf_{y\in S} h_{y,c}(x_0,kT+\frac{T}{2})\geqslant  \lim_{k\to +\infty} \inf_{y\in S} W^\epsilon_{y}(x_0,kT+\frac{T}{2})  
		=  \inf_{y\in S} W^\epsilon_{y}(x_0, \frac{T}{2}).
		\end{align*}

	We claim  that
	\begin{align}\label{1+1}
	 \inf_{y\in S} W^\epsilon_{y}(x_0, \frac{T}{2})>c.
	\end{align}
	Assume by contradiction that $ \displaystyle  \inf_{y\in S} W^\epsilon_{y}(x_0, \frac{T}{2})=c$. In view of \eqref{eq:eq:pf-thm1-step1-1},  we get
	\begin{align*}
		c= \inf_{y\in S} W^\epsilon_{y}(x_0, \frac{T}{2})=&\, c+\displaystyle \inf_{y\in S} \epsilon \sum_{i=1}^n \Big) 1+\sin \Big( 2\pi x^0_i- 2\pi y_i-\frac{\pi}{2} -2\pi  \omega_i \frac{T}{2}  \Big)  \Big)  \\
		 =&\, c+\displaystyle \inf_{y\in S} \epsilon \sum_{i=1}^n \Big( 1+\sin \Big( 2\pi x^0_i- 2\pi \big( y_i +   \omega_i \frac{T}{2} \big)  -\frac{\pi}{2}  \Big)  \Big) \\
		 =&\,c+\displaystyle \inf_{y\in S} \epsilon \sum_{i=1}^n \Big( 1+\sin \Big( 2\pi x^0_i- 2\pi \cdot \Pi \big( y_i +   \omega_i \frac{T}{2} \big)  -\frac{\pi}{2}  \Big)  \Big) \\
		 =&\, c+\displaystyle \inf_{z \in \Phi^H_{T/2} (S)} \epsilon \sum_{i=1}^n \Big( 1+\sin \Big( 2\pi x^0_i- 2\pi  z_i  -\frac{\pi}{2}  \Big)  \Big).
	\end{align*}
	This implies that 
	$$
	x_0\in    \Phi^H_{T/2} (S),
	$$
	and thus 
	$$
	x_0=  \Phi^H_{T/2} (S) \cap S\neq \emptyset,
	$$
	in contradiction to \eqref{eq:eq:pf-thm1-step1-1}. 
 So \eqref{1+1} holds true.

	On the other side, by Lemma  \ref{lem:lem-2}, we have
	$$
	\liminf_{t\to +\infty} \inf_{y\in S} h_{y,c}(x_0,t) \leqslant  \liminf_{t\to +\infty} h_{\Pi( x_0-\omega t ),c}(x_0,t) = c.
	$$  
	By \eqref{eq:pf-thm1-step2-2}, we get
	$$
	\liminf_{t\to +\infty} \inf_{y\in S} h_{y,c}(x_0,t) \geqslant \liminf_{t\to +\infty}  \inf_{y\in \T^n} h_{y,c}(x,t)  =c.
	$$
It follows that for any $x_0\in S$ 
		\begin{align*}
			\limsup_{t\to +\infty} \inf_{y\in S} h_{y,c}(x_0,t) \geqslant 	\lim_{k\to +\infty} \inf_{y\in S} h_{y,c}(x_0,kT+\frac{T}{2}) > c= \liminf_{t\to +\infty} \inf_{y\in S} h_{y,c}(x_0,t).
		\end{align*}
Therefore, $u(x,t)$ is  a non-trivial $T$-periodic viscosity  solution of \eqref{eq:HJe}.

\medskip

In the last step we show that if equation \eqref{eq:HJe} has a $T$-periodic viscosity solution then $T$ must belong to the set $D$. Denote by $w(x,t)$ a non-trivial $T$-periodic viscosity solution of equation \eqref{eq:HJe}.

\medskip

\noindent {\bf Step 5:}  Assume by contradiction that  $ T\notin D$.   Then for any  $( k_1,\cdots, k_n)\in \mathbb{Q}^n$, 
$$
(k_1\omega_1+ \cdots +k_n \omega_n )T\neq1.
$$

There are essentially two cases.

\medskip

\noindent (1) $\omega$ is rationally independent. We claim that 
\begin{equation}\label{eq:pf-thm1-step5--claim-0}
\text{
	$\omega_1T, \cdots ,\omega_n T $ and 1 are rationally independent}.
\end{equation}
Assume by contradiction  that there exists $k_1,\cdots k_{n+1}\in \mathbb{Z}^{n+1} \backslash \{0 \} $    such that 
\begin{equation}\label{eq:pf-thm1-step5-1}
		k_{n+1}=\sum_{i=1}^n k_i \omega_i T =  \Big(\sum_{i=1}^n k_i \omega_i \Big) T.
\end{equation}
Since  $\omega_1, \cdots ,\omega_n$ are rationally independent, then 
$\displaystyle \sum_{i=1}^n k_i \omega_i \neq 0$ and $k_{n+1}\neq 0$. Hence, we have
$$
T= \frac{k_{n+1}}{ \displaystyle \sum_{i=1}^n k_i \omega_i } = \Big( \sum_{i=1}^n \frac{k_i}{k_{n+1}} \omega_i \Big)^{-1}  \in D,
$$
which contradicts $T\notin D$. Thus \eqref{eq:pf-thm1-step5--claim-0} holds true.

By Proposition  \ref{Appendix:1} and  \eqref{eq:pf-thm1-step5--claim-0} one can deduce that the translation $\mathfrak{T}_\omega$  is minimal. 
By Proposition \ref{prop:WWY-longtime-JDE}, we know that $\{ x\in \T^n : w(x,0)=c\} \neq \emptyset $. Take any
$$
x_0\in \{ x\in \T^n : w(x,0)=c\}.
$$
Then by Lemma \ref{lem:lem-2}, for any $t>0$,
$$
w\Big(\Phi_t^H(x_0) , t\Big)= T_t^- w\Big(\Phi_t^H(x_0) ,0\Big)=c.
$$ 
In particular, for any $k\in \mathbb{Z}$,
\begin{equation}\label{eq:pf-thm2-(1)-1}
w\Big(\Phi_{kT}^H(x_0), 0\Big)=w\Big(\Phi_{kT}^H(x_0), kT\Big) = T_{kT}^- w\Big(\Phi_{kT}^H(x_0),0\Big)=c.
\end{equation} 
This implies that $w(x,0)=c$ for any $x\in \mathcal{O}_{\mathfrak{T}_\omega}(x)$. Since $\mathfrak{T}_\omega$  is minimal, we get
$$
w(x,0)=c, \quad a.e. \ x\in \T^n. 
$$
By the continuity of  $w(x,0)$, one follows 
$$
w(x,t)= T_t^- w(x,0)= T_t^- u_0 \equiv c, \quad \forall (x,t)\in \T^n\times[0,+\infty) .
$$ 
This contradicts that  $u(x,t)$ is non-trivial $T$-periodic solution of equation \eqref{eq:HJe}.

\medskip

\noindent (2) $\omega$ is rationally dependent.
Due to $ T\notin D$, in view of Proposition \ref{Appendix:3},  we deduce that for any $x_0\in \T^n$,
 	$$
 	 \bigcup_{n\in \mathbb{Z}}\Phi_{nT}^H(x_0)  \text{ is dense in }   \bigcup_{t\in \R} \Phi_{t}^H(x_0).
 	$$
 	 Let $x_0$ be such that  $w(x_0,0)=c$. Then similarly with \eqref{eq:pf-thm2-(1)-1}, we obtain that
 	$$
 	w\Big(\Phi_{nT}^H(x_0),0\Big)=c =w\Big(\Phi_{t}^H(x_0),0\Big), \quad \forall \, n \in \mathbb{N},\ \forall t>0.
 	$$
 	It follows that 
 	\begin{align}\label{0-0}
 	I_w := \{ x\in \T^n:  w(x,0)=c \}
 	\end{align}
 is $\Phi^H_t$-invariant.

 Applying Lemma \ref{lem:long time} below and  taking $\varphi=w(x,0)$, we have that 
 $$
 \liminf_{t\to +\infty }   w(x,t) =\limsup_{t\to +\infty }   w(x,t)= \lim_{t\to +\infty} \inf_{ y\in  \mathfrak{S}} h_{y,c}(x,t).
  $$
  This contradicts the assumption that $w(x,t)$ is a non-trivial $T$-periodic solution of equation \eqref{eq:HJe}.
  \end{proof}

 \begin{lemma} \label{lem:long time}
 	For any $\varphi \in C(\T^n,\R)$ satisfying  $\displaystyle \min_{x\in \T^n } \varphi = c$ and
 	$$
 		I_\varphi:=  \{ x\in \T^n:  \varphi(x)=c \}, \quad \Phi_t^H (I_\varphi)  = I_\varphi, \quad \forall t\geqslant 0,
 	$$ 
 there holds
 	$$
 	\lim_{t\to +\infty } T_t^- \varphi= \lim_{t\to +\infty} \inf_{ y\in  I_\varphi } h_{y,c}(x,t)\in \mathcal{S}^-.
 	$$
 \end{lemma}

 	\begin{proof}
 	We divide the proof  into three steps.
 	
 	\medskip
 	
 \noindent \textbf{Step 1:}  We assert that for any   $\epsilon>0$, there exists $t_1>0$ such that 
	$$
	T^-_t \varphi(x)= \inf_{y\in O_\epsilon}h_{y,\varphi(y)}(x,t) \quad \forall (x,t) \in \mathbb{T}^n \times [t_1,+\infty),
	$$
	where $O_\epsilon$ is the $\epsilon$-neighborhood of $I_\varphi$.

 	In fact, let 
	$$
	\sigma:=\min_{x\in \mathbb{T}^n \backslash O_\epsilon}(\varphi(x)-c).
	$$
	 Then by the definition of $I_\varphi$, $\sigma>0$ is well defined. Let $u_\sigma:= c+\sigma$. Then $\varphi(x)\geqslant u_\sigma (x)$, $\forall x\in \T^n \backslash O_\epsilon$. By (D3) in Proposition \ref{prop:WWY-longtime-JDE}, we have that 
	\begin{align}\label{8-8}
		\lim_{t\to+\infty}T^-_tu_\sigma(x)=+\infty,\quad \text{uniformly on}\ x\in \T^n .
	\end{align}
	From (D1) in Proposition \ref{prop:WWY-longtime-JDE} there exist  $K>0$ independent of $\varphi$ and $T_\varphi>0$ such that 
	$$
	|T^-_t \varphi(x) | \leqslant K , \quad \forall (x,t)\in \T^n \times (T_\varphi,+\infty).
	$$ 
	From \eqref{8-8}, there is $t_1>T_\varphi$ such that
	\[
	T^-_tu_\sigma(x)\geqslant  K +1,\quad \forall (x,t)\in \T^n \times (t_1,+\infty),
	\]
	where $t_1$ depends on $\epsilon$ and $\varphi$.
	Thus, for any $t\geqslant t_1$ and any $x\in \T^n$, we get 
	\begin{align}\label{7-4}
		\inf_{y\in \T^n \backslash O_\epsilon}h_{y,\varphi (y)}(x,t)\geqslant\, \inf_{y\in \T^n \backslash O_\epsilon}h_{y,u_\sigma(y)}(x,t) 
		\geqslant \, \inf_{y\in  \T^n}h_{y,u_\sigma(y)}(x,t)=T^-_tu_\sigma(x)\geqslant K +1.
	\end{align}
	Hence, for any $t\geqslant t_1$, any $x\in \T^n$, by \eqref{7-4} we have that
	\begin{align*}
		T^-_t\varphi (x)=\,
		\inf_{y\in \T^n }h_{y,\varphi (y)}(x,t) 
		=\,\min\big\{\inf_{y\in O_\epsilon}h_{y,\varphi (y)}(x,t),\min_{y\in \T^n \backslash O_\epsilon}h_{y,\varphi (y)}(x,t)\big\}
		=\inf_{y\in O_\epsilon}h_{y,\varphi (y)}(x,t).
	\end{align*}
	So far, we have shown the assertion.

\medskip

\noindent \textbf{Step 2:} 
On one hand, since $I_\varphi \subset O_\epsilon$, we get that
	$$
	T^-_t \varphi(x)= \inf_{y\in O_\epsilon}h_{y,\varphi(y)}(x,t)\leqslant \inf_{y\in I_\varphi }h_{y,\varphi(y)}(x,t)=\inf_{y\in I_\varphi }h_{y,c}(x,t),\quad (x,t)\in \T^n \times[t_1,+\infty).
	$$
	Thus, for any $(x,t) \in \T^n \times [0,+\infty)$, 
	$$
	T^-_{t} \varphi(x) \leqslant \inf_{y \in I_\varphi }h_{y,c}(x, t), \quad \forall t\geqslant t_1.
	$$
By Lemma \ref{lem:lem-2} we have that
$$
	h_{y,c}(x,t+s)  = \inf_{z\in \mathbb{T}^n} h_{z,h_{y,c}(z,s)}(x,t ) \leqslant h_{\Phi_s^H (y ),h_{y,c}\left(\Phi_s^H (y ),s\right)}(x, t ) =h_{\Phi_s^H (y ),c }(x, t ).
$$
Since $I_\varphi $ is $\Phi_t^H$-invariant, then

	\begin{align*}
	  \inf_{y\in  I_\varphi}h_{y,c}(x,t+s) \leqslant  \inf_{y\in I_\varphi}  h_{\Phi_s^H (y ) ,c}(x,t)  = \inf_{y\in \Phi^H_{-s}I_\varphi}  h_{y,c}(x,t) 
			  =\inf_{y\in I_\varphi}  h_{y,c}(x,t).
	\end{align*}	
One can deduce that
	\begin{align}\label{7-20}
		\limsup_{ t \to +\infty }  T^-_{t}\varphi(x) \leqslant \lim_{t\to +\infty} \inf_{y\in I_\varphi}h_{y,c}(x,t)=:U_1(x). 
	\end{align}

On the other hand, for $ t>t_1 $ we get that
	$$
	T^-_{t} \varphi(x) \geqslant \inf_{y\in \overline{O}_\epsilon }h_{y,c}(x, t),\quad \forall (x,t) \in \mathbb{T}^n \times [0,+\infty). 
	$$
	And thus,  
	\begin{align}\label{7-21}
		\liminf_{t\to +\infty}  T^-_{ t}\varphi(x) \geqslant \lim_{t\to +\infty} \inf_{y\in \overline{O}_\epsilon }h_{y,c}(x, t)=:U^\epsilon_2(x). 
	\end{align}
	By \eqref{7-20} and \eqref{7-21}, we have $U_1\geqslant U^\epsilon_2$.

\medskip

 \noindent \textbf{Step 3:} We prove that
  \begin{align}\label{7-26}
		\lim_{  \epsilon \to 0^+ } |U_1(x) - U^\epsilon_2(x)|=0, \quad\forall x\in \T^n. 
	\end{align}
Notice that for any $x_1 \in I_\varphi$ and $x_2\in  \overline{O}_\epsilon$ with $|x_1-x_2|\leqslant \epsilon$, by (H1)-(H3),
 $$
 h_{x_1,c}(x,t)= h_{x_2,c}(\Pi(x+x_2-x_1) ,t), \quad \forall (x,t)\in \T^n\times (0,+\infty).
 $$
 Then we have
\begin{align*}
	U^\epsilon_2(x):=&\,\lim_{t\to +\infty} \inf_{y\in \overline{O}_\epsilon }h_{y,c}(x, t) \geqslant  \inf_{|z|\leqslant \epsilon} \, \lim_{t\to +\infty} \inf_{y\in I_\varphi  }h_{y,c}(\Pi(x+z), t) = \inf_{|z|\leqslant \epsilon}   U_1(\Pi(x+z)).
\end{align*}
By the continuity of  $U_1(x)$ on $\T^n$, one gets that 	$\displaystyle \lim_{  \epsilon \to 0^+ } |U_1(x) - U^\epsilon_2(x)|=0$ for any $x\in \T^n$.
	Then we have 
	$$
	\liminf_{t\to +\infty}T^-_{t}\varphi(x)= \limsup_{t\to +\infty}T^-_{t}\varphi(x)
	$$
	 and thus $\lim_{t\to +\infty } T_t^- \varphi$ is a  viscosity solution of  $H(\partial_x u(x),u(x))=0$
	by Proposition \ref{prop:liminf}. 
\end{proof}

\section{Proof of Theorem \ref{thm:2}} 
 Assume $T\in \mathcal{D}$, we prove that equation \eqref{eq:HJe} admits   a non-trivial time-periodic solution with the fundamental period $T$ in this part.

Since $T\in \mathcal{D}$, in view of the proof of Theorem \ref{thm:1},  we only need to show that the $T$-periodic solution $u(x,t)$ defined in \eqref{eq:eq:pf-thm1-step3-def-u} satisfies
$$
u(x,t) \not \equiv u(x,0), \quad \forall\, t\in (0,T).
$$ 
First, we claim that
\begin{equation}\label{eq:pf-thm2-cliam-1}
\Phi_t^H (S)=\{x\in \T^n: u(x,t)=c \}, \quad \forall\, t\in \R.
\end{equation}
On one side, by Lemma \ref{lem:subsolution-w},  for $\epsilon\in (0,\epsilon_0] $, 
$$
u(x,t)=\lim_{k\to +\infty} U_k(x,t)=\lim_{k\to +\infty} \inf_{y \in S} h_{y,c}(x, kT+t) \geqslant \lim_{k\to +\infty} \inf_{y \in S} W^\epsilon_{y}(x,kT+ t).
$$
Since  $\displaystyle \inf_{y\in S} W^\epsilon_{y}(x,\cdot )$ is $T$-periodic in \eqref{eq:pf-thm1-step4-1}, then we have
\begin{align*}
		\lim_{k\to +\infty} \inf_{y \in S} W^\epsilon_{y}(x,kT+ t) =  &\, \inf_{y \in S} W^\epsilon_{y}(x,  t)  \\
	 =  &\,	 c+\displaystyle \inf_{y\in S} \epsilon \sum_{i=1}^n \Big( 1+\sin \Big( 2\pi x_i- 2\pi y_i-\frac{\pi}{2} -2\pi  \omega_i t \Big)  \Big) \\
		=&\, c+\displaystyle \inf_{y\in S} \epsilon \sum_{i=1}^n \Big( 1+\sin  \Big( 2\pi x_i -2\pi (y_i+\omega_i t) -\frac{\pi}{2}  \Big) \Big)\\
		=&\, c+\displaystyle \inf_{y\in S} \epsilon \sum_{i=1}^n \Big( 1+\sin  \Big( 2\pi x_i -2\pi\cdot \Pi(y_i+\omega_i t)-\frac{\pi}{2} \Big) \Big) \\
		=&\, c+\displaystyle \inf_{z  \in \Phi_{t}^H (S)} \epsilon \sum_{i=1}^n \Big( 1+\sin  \Big( 2\pi x_i -2\pi z_i -\frac{\pi}{2}  \Big) \Big) \\
		=&\, 	\displaystyle \inf_{y\in \Phi_{t}^H (S)} W^\epsilon_{y}(x, 0). 
\end{align*}
This implies that
\begin{equation}\label{eq:pf-thm2-cliam-2}
\{x\in \T^n: u(x,t)=c \} \subset  \{x\in \T^n :\displaystyle \inf_{y\in \Phi_{t}^H (S)} W^\epsilon_{y}(x, 0)=c \}= \Phi_{t}^H (S).
\end{equation}
 On the other side, by \eqref{eq:pf-step2-0} and \eqref{eq:pf-thm1-step2-2},  for any $x\in S$,
 $$
 c\leqslant u(x,0)= \lim_{k\to +\infty} \inf_{y \in S} h_{y,c}(x, kT) \leqslant \lim_{k\to +\infty} h_{\Phi_{-kT}^H(x) ,c}(x, kT) =c.
 $$
Thus,  for any $x_1\in \Phi_{t}^H (S)$, set $x_0:= \Phi_{-t}^H (x_1) \in S$,  by  Lemma \ref{lem:lem-2}, 
 $$
 u(x_1,t)=T_t^- u(x_1,0) = T_t^- u(\Phi_{t}^H(x_0) ,0)= c, \quad \forall\,  t\geqslant 0.
 $$
It follows that $  \Phi_{t}^H (S) \subset \{x\in \T^n: u(x,t)=c \}$. In view of \eqref{eq:pf-thm2-cliam-2}, we have proved  \eqref{eq:pf-thm2-cliam-1}.

Due to \eqref{eq:eq:pf-thm1-step1-1}, we get that the two nonempty sets $\{x\in \T^n: u(x,t)=c \}$ and $\{x\in \T^n: u(x,0)=c \}$  satisfy
$$
\{x\in \T^n: u(x,t)=c \} \bigcap  \{x\in \T^n: u(x,0)=c \}= \emptyset.
$$
So,
$$
u(x,t) \not \equiv u(x,0), \quad \forall\, t\in (0,T).
$$ 
Hence, $T$ is the fundamental period of the solution $u(x,t)$.

\section{Proof of Corollary \ref{Cor:1}}
In this section, we focus on the construction of almost time-periodic viscosity solutions. We collect some properties of almost periodic functions here.
\begin{itemize}
	\item If $f$ is almost periodic, then $|f|$  is also almost periodic;
	\item If $f$ and $g$ are almost periodic, then $f+g$, $f-g$  are also almost periodic.
\end{itemize}

Hence, the proof of Corollary \ref{Cor:1} is a direct consequence of the following Lemma:
\begin{lemma}\label{lem:almost-periodic}
	For any fixed $\mathcal{T}_1, \mathcal{T}_2\in D$ with $\frac{\mathcal{T}_1}{\mathcal{T}_2}\notin \mathbb{Q}$, there exist two non-trivial time-periodic viscosity solutions $w_1(x,t)$, $w_2(x,t)$  of equation \eqref{eq:HJe} with period  $\mathcal{T}_1$ and $\mathcal{T}_2$, respectively.  Furthermore,
	$$
			w(x,t):=\min\{ w_1(x,t), w_2(x,t) \}
			$$
			is a non-trivial almost periodic viscosity  solution of equation \eqref{eq:HJe}.
\end{lemma}
\medskip

\begin{proof}
In order to give the proof clearly, we divide it into three parts as follows.

\noindent {\bf Step1:} We show that $w(x,t)$ is a viscosity  solution of \eqref{eq:HJe}. It is sufficient to show the following lemma.

\begin{lemma}\label{lem:Tt-+inf}
The following two properties of viscosity solutions hold.\begin{itemize} 
	\item [(1)] Let $\{\varphi_i(x)\}_{i\in I}$ be a family of continuous functions on $\T^n$. Then for each $t>0$, 
	$$
	\inf_{i\in I} T_t^{-} \varphi_i (x)=T_t^- \big( \inf_{i\in I} \varphi_i(x) \big), \quad \forall x\in \T^n.
	$$
	\item [(2)] Let $\{u_i(x,t)\}_{i\in I}$ be a family of viscosity solutions to equation \eqref{eq:HJe} on $\T^n\times \R$. Then 
	$$
	u(x,t):=\inf_{i\in I} \,u_i(x,t) ,\quad (x,t)\in \T^n\times \R
	$$
	is a viscosity solution of \eqref{eq:HJe}.
\end{itemize}
	\end{lemma}
	\begin{proof}
	To prove (1), let us recall 
	$$
	T_t^{-} \big( \inf_{i\in I} \varphi_i(x) \big)= \inf_{y\in \T^n} h_{y, \inf_{i\in I}\varphi _i(y)} (x,t).
	$$
	Due to the monotonicity of $h_{y,u}(x,t)$ with respect to $u$, we get that
	$$
	\inf_{y\in \T^n} h_{y, \inf_{i\in I}\varphi _i(y) }(x,t)= \inf_{y\in \T^n}\inf_{i\in I} h_{y, \varphi _i(y) }(x,t).
	$$
	Since $y$  is independent of $i$ , then
	$$
	 \inf_{y\in \T^n}\inf_{i\in I} h_{y, \varphi _i(y) }(x,t)= \inf_{i\in I}\inf_{y\in \T^n} h_{y, \varphi _i(y) }(x,t)=\inf_{i\in I} T_t^{-} \varphi_i(x),
	$$
	which implies $ T_t^{-} \big( \inf_{i\in I} \varphi_i(x) \big)= \inf_{i\in I} T_t^{-} \varphi_i(x)$.

		To prove (2), for any $t_1 < t_2 $, we have
	\begin{align*}
		u(x,t_2)= \inf_{i\in I} \,u_i(x,t_2)=\inf_{i\in I}\, T_{t_2-t_1}^- u_i(x,t_1) 
		= T_{t_2-t_1}^-\inf_{i\in I}\, u_i(x,t_1) = T_{t_2-t_1}^-u(x,t_1),
	\end{align*}
	which completes the proof.
	\end{proof}

\medskip

\noindent	{\bf Step2:} We show that $w(x,t)$ is an almost periodic solution of \eqref{eq:HJe}.
Assume that $w_1(x,t), w_2(x,t) $ are two time-periodic viscosity solutions  of \eqref{eq:HJe}. By the properties of almost periodic functions, it is clear  that
$$
	\min\{ w_1(x,t), w_2(x,t) \}=\frac{1}{2}\Big(  w_1(x,t)+ w_2(x,t)-|w_1(x,t)- w_2(x,t)|\Big). 
	$$	
	is an almost periodic viscosity solution of \eqref{eq:HJe}.

\medskip

\noindent{\bf Step3:} We show that $w(x,t)$ is not a periodic solution of \eqref{eq:HJe}. 
For $\mathcal{T}_i \in D,i=1,2$, there exist $( r_1,\cdots, r_{n+1}),(s_1,\cdots, s_{n+1}) \in \mathbb{Z}^n \backslash\{0\} $ such that 
$$
(r_1\omega_1+ \cdots +r_n \omega_n ) \cdot r^{-1}_{n+1} \mathcal{T}_1=1 , \quad (s_1\omega_1+ \cdots +s_n \omega_n ) \cdot s^{-1}_{n+1} \mathcal{T}_2=1 
$$
	Set 
	\begin{align*}
		S_{\mathcal{T}_1}: =\Pi \Big( \{ \, x\in \R^n :  r_1x_1+ \cdots +r_n x_n \in \mathbb{Z} \}\Big), \\
			S_{\mathcal{T}_2}: =\Pi \Big( \{ \, x\in \R^n :  s_1x_1+ \cdots +s_n x_n \in \mathbb{Z} \}\Big),
	\end{align*}
	 and
	$$
	w_i(x,t):= \lim_{k\to +\infty }\inf_{y\in S_{\mathcal{T}_i}} h_{y,c}(x,kT+t), \quad i=1,2, \quad (x,t)\in \mathbb{T}^n\times \R^+. 
	$$
Then
		$$
			w(x,t):=\min\{ w_1(x,t), w_2(x,t) \}=  \lim_{k\to +\infty }\min_{y\in S_{\mathcal{T}_1}\cup S_{\mathcal{T}_2}  } h_{y,c}(x,kT+t).
			$$
By Lemma \ref{lem:lem-2} and Lemma \ref{lem:subsolution-w}, one gets
 $$
 \{x\in \mathbb{T}^n: w(x, 0 )=c \}= S_{\mathcal{T}_1}\cup S_{\mathcal{T}_2}.
 $$

	Assume by contradiction that $w(x,t)$ is a  time-periodic viscosity solution of equation  \eqref{eq:HJe} with some periodic $  \mathcal{T}$. Then  by Lemma \ref{lem:lem-2}
	
$$
 S_{\mathcal{T}_1}\cup S_{\mathcal{T}_2}=\Phi^H_\mathcal{T}  \Big( S_{\mathcal{T}_1}\cup S_{\mathcal{T}_2} \Big) =\Phi^H_\mathcal{T}( S_{\mathcal{T}_1})\cup \Phi^H_\mathcal{T}(S_{\mathcal{T}_2}). 
$$
Note that $\Phi^H_\mathcal{T} (S_{\mathcal{T}_1})\subset  S_{\mathcal{T}_1}\cup S_{\mathcal{T}_2} $   implies $ \frac{ \mathcal{T}}{ \mathcal{T}_1} \in \mathbb{Q} $ and $\Phi^H_\mathcal{T} (S_{\mathcal{T}_2})\subset  S_{\mathcal{T}_1}\cup S_{\mathcal{T}_2} $ implies $ \frac{ \mathcal{T}}{ \mathcal{T}_2} \in \mathbb{Q} $. It follows
	$$
	\frac{\mathcal{T}_1}{\mathcal{T}_2}\in \mathbb{Q},
 $$
 a contradiction.
\end{proof}

\begin{proof}[Proof of Corollary \ref{Cor:1}]
We divide the proof into two steps.
 
 \noindent {\bf Step 1:} We show that if there exist $i$, $j \in \{1,2,\cdots,n \}$  such that $\omega_i/\omega_j \notin \Q $, then equation \eqref{eq:HJe} admits infinitely many non-trivial Bohr almost periodic viscosity solutions.
 
 Take
 $$
 \mathcal{T}_k:= w_i +k w_j, \quad k\in \mathbb{N} .
 $$ 
 It is clear that $\mathcal{T}_k \in \mathcal{D} \subset D$ and $\mathcal{T}_{k_1}/ \mathcal{T}_{k_2} \notin \Q$ for each $k_1,k_2\in \mathbb{N} $. In view of  Lemma \ref{lem:almost-periodic}, we deduce  that  there exist  non-trivial time-periodic viscosity solutions $w_k(x,t)$  of equation \eqref{eq:HJe} with period  $\mathcal{T}_k$ for each $k\in \mathbb{N} $.  Furthermore,
	$$
			\min\{ w_1(x,t), w_k(x,t) \}, \quad  k \in \mathbb{N}
			$$
			are different non-trivial almost periodic viscosity  solutions of equation \eqref{eq:HJe}. Hence, equation \eqref{eq:HJe} admits infinitely many non-trivial almost periodic viscosity solutions.

\medskip			
 \noindent {\bf Step 2:} We show that if  equation \eqref{eq:HJe} admits infinitely many non-trivial  almost periodic viscosity solutions, then 		there exist $i$, $j \in \{1,2,\cdots,n \}$  such that $\omega_i/\omega_j \notin \Q $.
 
Assume by contradiction that for any $i,j \in \{1,2,\cdots,n \},w_j=0 $ or  $\omega_i/\omega_j \in \Q$. Assume that $\omega_1\neq 0$. There exists $k_0 \in \mathbb{N}$ such that 
 $$
k_0 \frac{ \omega_k}{\omega_1}   \in \mathbb{Z}, \quad \forall \, k \in \{1,2,\cdots,n \}.
 $$
Set $\mathcal{T}:=\frac{k_0}{\omega_1}$.  This implies that $ \Phi_{ \mathcal{T}}^H (x_0)=x_0$ for any $x_0\in \T^n$. Thus, for any  $\varphi \in C(\T^n,\R)$ satisfying   $\displaystyle \min_{x\in \T^n } \varphi = c$, we get that
 	$$
 \Phi_{\mathcal{T} } ^H (I_\varphi)  = I_\varphi, 
 	$$ 
 	where $	I_\varphi:=  \{ x\in \T^n:  \varphi(x)=c \}$.

 	One can deduce the following results (i)-(iv). The proofs of (i)-(iii) are similar to the ones in the proof of Lemma \ref{lem:long time}.  The item (iv) can be proved by step 3 in the proof of Theorem \ref{thm:1}. 	
 	
 	\begin{itemize}
 		\item[(i)] for any   $\epsilon>0$, there exists $t_1>0$ such that 
	$$
	T^-_t \varphi(x)= \inf_{y\in O_\epsilon}h_{y,\varphi(y)}(x,t) \quad \forall (x,t) \in \mathbb{T}^n \times [t_1,+\infty).
	$$
     \item [(ii)] For any $(x,t) \in \T^n \times [t_1,+\infty)$, 
	$$
\lim_{n\to +\infty}  \inf_{y\in \overline{O}_\epsilon }h_{y,c}(x,n\mathcal{T}+ t)=:U^\epsilon_2(x,t) \leqslant T^-_{t}  \varphi(x) \leqslant U_1(x,t):=\lim_{n\to +\infty}  \inf_{y \in I_\varphi }h_{y,c}(x,n\mathcal{T}+ t).
	$$
\item[(iii)]For any $(x,t) \in \T^n \times [0,+\infty)$, $\displaystyle \lim_{  \epsilon \to 0^+ } |U_1(x,t) - U^\epsilon_2(x,t)|=0 $. \\ 
\item[(iv)] For any $(x,t) \in \T^n \times [0,+\infty)$,
  $$
  \lim_{n\to +\infty} T_{n\mathcal{T}+ t}^- \varphi(x)=U_1(x,t)  
  $$
   and $U_1(x,t)$ is a $\mathcal{T}$-periodic solution of equation \eqref{eq:HJe}.
 	\end{itemize}
\medskip 	
 	Hence, we conclude that  for any  $\varphi \in C(\T^n,\R)$ satisfying   $\displaystyle \min_{x\in \T^n } \varphi = c$, the limit
 	$$
 	  \lim_{n\to +\infty} T_{n\mathcal{T}+ t}^- \varphi(x):= w_\infty^\varphi(x,t)   \quad \forall (x,t) \in \T^n \times [0,+\infty)
 	$$
 	exists and $w_\infty^\varphi(x,t)$ is a time periodic viscosity solution of  equation  \eqref{eq:HJe} with periodic $\mathcal{T}$.
 	This contradicts that there exists a non-trivial  almost periodic viscosity solution of  equation   \eqref{eq:HJe}.
  \end{proof}

\section{An example}
In this section we provide a simple example to illustrate our results.
Consider the contact Hamiltonian
$$
H(x,p,u)=\langle p,p \rangle +\langle \omega,p\rangle -\lambda u, \quad \lambda>0, \ \omega\neq 0, \  \ x\in \mathbb{T}^n.
$$
Let $x=(x_1,x_2,\cdots,x_n)$, $\omega= (\omega_1,\omega_2,\cdots \omega_n)$ and $x_0=(x_1^0,x_2^0,\cdots x_n^0)$.
It is clear that for any $x_0\in \T^n$,  
$$
w_{x_0}(x,t):=\min_{k\in\mathbb{Z}}\frac{\lambda}{4}( x_1- x_1^0+k-\omega_1 t)^2 + \cdots +\min_{k\in\mathbb{Z}}\frac{\lambda}{4}( x_n- x_n^0+k-\omega_n t)^2
$$
 is a viscosity solution of \eqref{eq:HJe} satisfying $w_{x_0}(x_0,0)=0$. 
 For any $T\in D$, there exists $( k_1,\cdots, k_{n+1})\in \mathbb{Z}^n$ such that $(k_1\omega_1+ \cdots +k_n \omega_n )T=k_{n+1}\neq 0$.
 Set  
 \begin{align*}
 	u(x,t)=&\, \min_{x_0\in S_T} w_{x_0}(x,t) \\
 	=&\, \min_{x_0\in S_T}\Big( \min_{k\in\mathbb{Z}}\frac{\lambda}{4}( x_1- x_1^0+k-\omega_1 t)^2 + \cdots +\min_{k\in\mathbb{Z}}\frac{\lambda}{4}( x_n- x_n^0+k-\omega_n t)^2 \Big),
 \end{align*}
 where $S_T:=\Pi( l)=\Pi\big(\{ \,x \in \R^n :  k_1x_1+ \cdots +k_n x_n \in \mathbb{Z} \}\big)  $.  Then for any $(x,t)\in \T^n\times [0,+\infty)$,
 \begin{align*}
 	T_T^-u(x,t)=&\,  T_T^-\min_{x_0\in S_T} w_{x_0}(x,t)\\ =&\,\min_{x_0\in S_T}  T_T^-w_{x_0}(x,t) \\
 	=&\,\min_{x_0\in S_T} w_{x_0}(x,t+T) \\
 	=&\, \min_{x_0 \in S_T} w_{x_0+\omega T}(x,t)\\=&\, \min_{x_0 \in \Phi^H_T(S_T)} w_{x_0}(x,t)\\=&\,u(x,t).
 \end{align*}
 Thus $u(x,t)$ is a non-trivial $T$-periodic solution of equation \eqref{eq:HJe}.

Moreover, 	there exist $i$, $j \in \{1,2,\cdots,n \}$  such that $\omega_i/\omega_j \notin \Q $ , then there exist   $\mathcal{T}_1,\mathcal{T}_2\in D$ satisfying $\mathcal{T}_1/\mathcal{T}_2\notin \Q $, by Lemma \ref{lem:almost-periodic},   
\begin{align*}
\widetilde u(x,t):=&\,\min \Big\{  \min_{x_0\in S_{\mathcal{T}_1}} w_{x_0}(x,t),\min_{x_0\in S_{\mathcal{T}_2}} w_{x_0}(x,t)  \Big\} \\
=&\,\min_{x_0\in S_{\mathcal{T}_1}\cup S_{\mathcal{T}_2} } w_{x_0}(x,t)
\end{align*}
is a non-trivial almost periodic viscosity solution of equation \eqref{eq:HJe}.

  \bigskip

\noindent   {\bf There is no conflict of interest and no data in this paper.}
  
  	\medskip

  \section*{Acknowledgements}
  Kaizhi Wang is supported by National Key R\&D Program of China (2022YFA1005900), NSFC Grant Nos. 12171315, 11931016 and Natural Science Foundation of Shanghai No. 2ZR1433100. Jun Yan is supported by NSFC Grant Nos. 12171096,  12231010. Kai Zhao is supported by NSFC Grant No. 12301233.

\end{document}